\documentclass[a4paper]{amsart}

\usepackage[margin=3cm]{geometry}

\usepackage{enumerate}

\usepackage[T1]{fontenc}
\usepackage[utf8]{inputenc}

\usepackage{amsmath}
\usepackage{amssymb}
\usepackage{amsthm}

\title[Biseparable vs. Frobenius]{Biseparable extensions are not necessarily Frobenius}

\thanks{Research supported by the grant MTM2016-78364-P from Agencia Estatal de Investigaci{\'o}n and FEDER. The fourth author was supported by The National Council of Science and Technology (CONACYT) of Mexico with a scholarship for a Postdoctoral Stay in the University of Granada.}

\author[G{\'{o}}mez-Torrecillas]{Jos{\'{e}} G{\'{o}}mez-Torrecillas}
\address{Department of Algebra, University of Granada, E-18071, Granada, Spain}
\email{gomezj@ugr.es}
\author[Lobillo]{F. J. Lobillo}
\address{CITIC and Department of Algebra, University of Granada, E-18071, Granada, Spain}
\email{jlobillo@ugr.es}
\author[Navarro]{Gabriel Navarro}
\address{CITIC and Department Computer Sciences and AI, University of Granada, E-18071, Granada, Spain}
\email{gnavarro@ugr.es}
\author[S{\'{a}}nchez-Hernández]{Jos{\'{e}} Patricio S{\'{a}}nchez-Hern{\'{a}}ndez}
\address{Department of Algebra, University of Granada, E-18071, Granada, Spain}
\email{jpsanchez@correo.ugr.es}

\theoremstyle{plain}
 \newtheorem{theorem}{Theorem}[section]
 \newtheorem{lemma}[theorem]{Lemma}
 \newtheorem{proposition}[theorem]{Proposition}
 \newtheorem{corollary}[theorem]{Corollary}

\theoremstyle{definition}
 \newtheorem{definition}[theorem]{Definition}
 \newtheorem{example}[theorem]{Example}
 \newtheorem{problem}[theorem]{Problem}
 
\theoremstyle{remark}
 \newtheorem{remark}[theorem]{Remark}
 
\newcommand{\field}{\mathbb{F}}

\begin{document}

\begin{abstract}
We give necessary and sufficient conditions on an Ore extension $A[x;\sigma,\delta]$, where $A$ is a finite dimensional algebra over a field $\field$, for being a Frobenius extension of the ring of commutative polynomials $\field[x]$. As a consequence, as the title of this paper highlights, we provide a negative answer to a problem stated by Caenepeel and Kadison.
\end{abstract}

\maketitle 

\section{Introduction}\label{sec:intro}

Frobenius extensions were introduced by Kasch \cite{Kasch:1954,Kasch:1961}, and by Nakayama and Tsuzuku \cite{Nakayama/Tsuzuku:1959,Nakayama/Tsuzuku:1960} as a generalization of the well known notion of Frobenius algebra. Of course the underlying idea was to recover the duality theory of Frobenius algebras in a more general setting. The notion of separable extension comes from the generalization of the well known notion of separable field extension. The classical definition of separable ring extension is due to Hirata and Sugano in \cite{Hirata/Sugano:1966}. Both notions, Frobenius and separable, have been extended to more general framework in category theory. 

As it is explained in the Introduction of \cite{Caenepeel/Kadison:2001}, deep connections between separable and Frobenius extensions were found from the very beginning. For instance, Eilenberg and Nakayama show in \cite{Eilenberg/Nakayama:1955} that finite dimensional semisimple algebras over a field are symmetric, hence Frobenius. A key result to extend this to algebras over commutative rings is due to Endo and Watanabe, concretely they show that separable, finitely generated, faithful and projective algebras over a commutative ring are symmetric \cite[Theorem 4.2]{Endo/Watanabe:1967}. Their ideas were connected to separable extensions, as defined in \cite{Hirata/Sugano:1966}, by Sugano, who shows that separable and centrally projective extensions are Frobenius, see \cite[Theorem 2]{Sugano:1970}. However, as Caenepeel and Kadison say ``it is implicit in the literature that there are several cautionary examples showing separable extensions are not always Frobenius extensions in the ordinary untwisted sense''. They provide one of these examples in \cite[\S 4]{Caenepeel/Kadison:2001} under the stronger hypothesis that the extension is split, but the Frobenius property is lost because the provided extension is not finitely generated. Split extensions are naturally considered since separability and splitting can be viewed as particular cases of the notion of separable module introduced in \cite{Sugano:1971}, see also \cite{Kadison:1996}. Biseparable extensions are therefore considered because they contains both notions of separable and split extensions under the same module theoretic approach. Biseparable extensions are finitely generated and projective, hence the example they provide is not a counter example of their main question: ``Are biseparable extensions Frobenius?''

This problem comes up again recently in the article \cite{Kadison:2019}, whether additional convenient equations are always satisfied: this is the same as asking if a biseparable bimodule is Frobenius. There are arguments in the monograph \cite{Kadison:1999} as evidence for thinking this might be true, as well as the weight of all classical examples.

In this paper we develop some techniques based in the Ore extensions introduced in \cite{Ore:1933} to provide a counter example to the previous question. Our example also gives a negative answer the same question but considering Frobenius extensions of the second kind as introduced by Nakayama and Tsuzuku in \cite{Nakayama/Tsuzuku:1960}. 

This paper is structured as follows. In section \ref{sec:preliminaires}, we recall precise definitions of Frobenius and biseparable extensions, and we recall again the main question we are going to answer. In section \ref{Frobenius} Frobenius extensions are lifted under Ore extensions, while similar results are obtained in section \ref{Biseparable} for biseparable extensions. Finally, in section \ref{sec:Example} the full counter example is built.

\section{Preliminaries}\label{sec:preliminaires}

We recall the notions of Frobenius, separable and split extensions. All along the paper $B$ and $C$ are arbitrary unital rings, whilst we reserve the letter $A$ for denoting an algebra over a field $\mathbb{F}$. Following, for instance, \cite{Nakayama/Tsuzuku:1959}, a unital ring extension $C \subseteq B$ is said to be Frobenius if $B$ is a finitely generated projective right $C$-module and there exists an isomorphism  $B \cong B^*=\operatorname{Hom}(B_C,C_C)$ of \(C-B\)-bimodules. Here, by $\operatorname{Hom}(B_C,C_C)$, we denote the set of morphisms of right $C$-modules from $B$ to $C$. The additive group $B^*$ is endowed with the standard \(C-B\)-bimodule structure given by $(c \chi b)(u)= c (\chi(bu))$ for any $\chi\in B^*$, $c\in C$ and $b,u\in B$. 

The notion of a Frobenius extension is right-left symmetric as observed in \cite[\S 1, page 11]{Nakayama/Tsuzuku:1959}, i.e. \(C \subseteq B\) is Frobenius if \(B\) is a finitely generated projective left \(C\)-module and there exists an isomorphism \(B \cong {}^*B\) of \(B-C\)-bimodules, where \({}^*B = \operatorname{Hom}({}_CB,{}_CC)\) is a \(B-C\)-bimodule in a analogous way.

This is a generalization of the well-known notion of Frobenius algebra over a field, namely, a finite dimensional $\field$-algebra $A$ is Frobenius if the following equivalent conditions hold:
\begin{enumerate}
\item there exists an isomorphism of right (or left) $A$-modules $A \cong A^*$ \label{FrobAlgIso}
\item there exists an associative and non-degenerate $\field$-bilinear form $\langle -,-\rangle:A\times A\to \field$ \label{FrobAlgBilin}
\item there exists a linear functional $\varepsilon: A \to \field$ whose kernel does not contain a non zero right (or left) ideal.\label{FrobAlgLin}
\end{enumerate}

\begin{remark}\label{FrobAlg}
The bijection between Frobenius forms (\ref{FrobAlgBilin}) and Frobenius functionals (\ref{FrobAlgLin}) on $A$ is as follows. If $\langle -,-\rangle: A \times A \to \field$ is a Frobenius form, then the rule $\varepsilon(a) = \langle 1,a \rangle$ for any $a\in A$ defines a Frobenius functional $\varepsilon:A\to \mathbb{F}$. Conversely, if $\varepsilon:A\to \field$ is a Frobenius functional, set $\langle a,b\rangle=\varepsilon(ab)$ for any $a,b\in A$ in order to get a Frobenius form. 

The correspondence between Frobenius functionals (\ref{FrobAlgLin}) and left $A$-isomorphisms (\ref{FrobAlgIso}) is given as follows. For any Frobenius functional $\varepsilon$, we may define $\alpha:A\to A^*$ as
$\alpha(a)(b) = \varepsilon(ab)$
for any $a,b\in A$, which becomes a left $A$-isomorphism. Conversely, for any left $A$-isomorphism $\alpha:A\to A^*$, the rule $\varepsilon(a)=\alpha(a)(1)$ for any $a\in A$ provides a Frobenius functional $\varepsilon$. See \cite[Theorem 3.15]{Lam:1999} for full details. In particular, for each \(\field\)-basis \(\{a_1, \dots, a_r\}\) of \(A\) there exists an \(\field\)-basis \(\{b_1, \dots, b_r\}\) of \(A\) such that \(\{\alpha(b_1), \dots, \alpha(b_r)\}\) is the dual basis of \(\{a_1, \dots, a_r\}\), i.e.
\begin{equation}\label{dualbasis}
\varepsilon(b_j a_i) = \alpha(b_j)(a_i) = \delta_{ij}.
\end{equation}
\end{remark}

Following \cite{Hirata/Sugano:1966}, the extension \(C \subseteq B\) is called separable if the canonical multiplication map
\[
\begin{split}
\mu : B \otimes_C B &\to B \\
b_1 \otimes b_2 &\mapsto b_1 b_2
\end{split}
\]
splits as a morphism of \(B\)-bimodules, i.e. there exists \(p \in B \otimes_C B\) such that \(bp = pb\) for all \(b \in B\) and \(\mu(p) = 1\). The splitting map is therefore determined by \(1 \mapsto p\). 

Finally, \(C \subseteq B\) is called split if the inclusion map \(C \to B\) splits as a morphism of \(C\)-bimodules, i.e. there exists a \(C\)-bimodule morphism \(\xi : B \to C\) such that \(\xi(1) = 1\). 

In \cite[Definition 2.4]{Caenepeel/Kadison:2001}, the notion of a separable module is extended to the concept of biseparable module. When particularizing to ring extensions, \cite[Lemma 3.3]{Caenepeel/Kadison:2001} says that \(C \subseteq B\) is called to be biseparable if one of the following equivalent conditions holds:
\begin{enumerate}
\item $B$ is biseparable as \(B-C\)-bimodule and finitely generated projective as left $C$-module.
\item $B$ is biseparable as \(C-B\)-bimodule and finitely generated projective as right $C$-module.
\item $B$ is biseparable as \(B-C\)-bimodule and as \(C-B\)-bimodule.
\item $C \subseteq B$ is split, separable and finitely generated projective as left $C$-module and as right $C$-module.
\end{enumerate}

Henceforth, motivated by the arguments provided in the Introduction, the following question is stated in \cite{Caenepeel/Kadison:2001}:

\begin{problem}\label{problem}\cite[Problem 3.5]{Caenepeel/Kadison:2001}\label{theproblem}
Are biseparable extensions Frobenius?
\end{problem}

The main aim of this paper is to build an example of a ring extension which is biseparable and not Frobenius, giving a negative answer to Problem \ref{problem}. Throughout the paper we assume that $A$ is a finite dimensional $\field$-algebra of dimension $r$. Let also denote by $\sigma:A \to A$ an algebra $\field$-automorphism and $\delta: A \to A$ a $\sigma$-derivation on $A$, i.e. \(\delta(ab) = \sigma(a) \delta(b) + \delta(a) b\) for all \(a,b \in A\). We denote by  $R$ the ring of (commutative) polynomials $\field[x]$ and by $S$ the Ore extension $A[x;\sigma,\delta]$, that is, the ring of polynomials with coefficients in $A$ written on the left whose product is twisted by the rule $xa=\sigma(a)x+\delta(a)$ for any $a\in A$. This notation is fixed throughout the rest of the paper. 

We give conditions on \(\sigma\) and \(\delta\) in order to get that \(R \subseteq S\) inherits the corresponding properties (separable, split, Frobenius) from \(\field \subseteq A\). A precise construction of \(A\), \(\sigma\) and \(\delta\) will lead to the counterexample.

\section{Lifting Frobenius extensions}\label{Frobenius}

Given $a\in A$, $n\geq 0$ and $0\leq i\leq n$, we denote by $N_i^n(a)$ the coefficient of degree $i$ when multiplying $x^n$ on the right by $a$ in $S$. That is to say, 
\begin{equation}\label{N}
x^na=\sum_{i=0}^nN_i^n(a)x^{i}, 
\end{equation}
and, for $\sum_{i= 0}^n g_ix^i \in S$, 
\begin{equation}\label{ga}
\left(\sum_{i = 0}^ng_ix^i\right)a = \sum_{i=0}^n \left(\sum_{k= i}^n g_kN_i^k(a)\right)x^i.
\end{equation}
We may then consider \(\field\)-linear operators $N_i^n:A \to A$ for every $i$ and $n$ with $0\leq i\leq n$. If we set $N_{i}^n=0$ whenever $i<0$ or $i > n$, then we obtain inductively
\begin{equation}\label{Nnmas1}
N_i^{n+1} = \sigma N_{i-1}^n + \delta N_i^n. 
\end{equation}

These maps were introduced in \cite{Lam/Leroy:1988}, where \(N_i^n\) is denoted by \(f_i^n\).  

The ring extension $R \subseteq S$ makes $S$ free of finite rank both as a left as a right $R$--module. More precisely, we have the following result.

\begin{lemma}\label{lemma 1a}  
Let $\{a_1,\ldots,a_r\}$ be an $\field$-basis of $A$. The following statements hold.
\begin{enumerate}
\item $\{a_1,\ldots ,a_r\}$ is a right basis of $S$ over $R$.
\item $\{a_1,\ldots,a_r\}$ is a left basis of $S$ over $R$.
\end{enumerate}
\end{lemma}

\begin{proof} 
\emph{(1)} This is an easy computation. 

\emph{(2)} It is well known that \(S^{op} = A^{op}[x;\sigma^{-1},-\delta \sigma^{-1}]\) (see e.g. \cite[page 39, Exercise 2R]{Goodearl/Warfield:2004}).  Now, apply part (1) to $S^{op}$. 
\end{proof}

By $\{a_1^*, \dots, a_r^* \}$ we will denote the basis of the left $R$--module $S^*$ dual to an \(R\)-basis $\{a_1, \dots, a_r \}$ of \(S\) as right \(R\)-module, determined by the condition $a_i^*(a_j) = \delta_{ij}$.

The aim of this section is to characterize when $R \subseteq S$ is a Frobenius ring extension in terms of the $\sigma$--derivation $\delta$ acting on $A$. The key result to get such a characterization is the following theorem.

\begin{theorem}\label{semiFrobext} 
There exists a bijective correspondence between the following sets.
\begin{enumerate}
\item Frobenius functionals on the $\field$-algebra $A$.
\item Right $S$-isomorphisms from $S$ to $S^*$.
\end {enumerate}
\end{theorem}
\begin{proof}
Let $\varepsilon:A \rightarrow \field$ be a Frobenius functional on $A$.  

To define a right $S$--linear map $\alpha_\varepsilon:S\to S^*$ we just need to specify $\alpha_\varepsilon (1) \in S^*$.  For every $f=\sum_i f_ix^i\in S$, set 
\[
\alpha_\varepsilon(1)(f)=\sum_i \varepsilon(f_i)x^i.
\]

This map is indeed right $R$--linear, since
\[
\alpha_\varepsilon(1)(f x ) = \sum_{i} \varepsilon(f_i) x^{i+1}  = \alpha_\varepsilon(1)(f)x. 
\]
Note that, by the right $S$--module structure of $S^*$, one has, for every $f, g \in S$,
\[
\alpha_\varepsilon (f)(g) = \alpha_\varepsilon (1) (fg) = \alpha_\varepsilon(fg)(1). 
\]

Let $f=\sum_{i=0}^nf_ix^{i}\in S$ with $f_n\neq 0$ such that $\alpha_\varepsilon(f)=0$. Then, for any $b\in A$, we get from \eqref{ga} that
\[
0=\alpha_\varepsilon(f)(b)= \alpha_\varepsilon(fb)(1) = \sum_{i=0}^n\varepsilon\left(\sum_{k=i}^nf_kN_i^k(b)\right)x^i.
\]
In particular, $\varepsilon(f_nN_n^n(b))=\varepsilon(f_n\sigma^n(b))=0$  for every $b\in A$. Since $\sigma$ is an automorphism,  $\varepsilon(f_nb) =0$ for all $b\in A$ and, thus, the kernel of $\varepsilon$ contains the right ideal generated by $f_n$, a contradiction. Thus $\alpha_\varepsilon$ is injective.

Finally, it remains to prove that $\alpha_\varepsilon$ is surjective. Let $\{a_1, \dots ,a_r\}$ be an $\field$-basis of $A$.  Let us show that $x^na_i^* \in \operatorname{Im} \alpha_{\epsilon}$ for all $n\geq 0$ and $1\leq i \leq r$, which yields the result.

For any $n\geq 0$, since $\{\sigma^{n}(a_1), \ldots , \sigma^n(a_r)\}$ is an $\field$-basis of $A$, by \eqref{dualbasis}, there exist $b^{(n)}_1,\ldots,b^{(n)}_r\in A$ such that  
\begin{equation}\label{eq bi ai}
\varepsilon\left( b^{(n)}_{i} \sigma^n(a_j)\right) = \delta_{ij}
\end{equation}
for all $1\leq i,j \leq r$. For each \(1 \leq i \leq r\), set
\[
g^{(i)}=\sum_{k=0}^ng_{k}^{(i)} x^k\in S,
\]
where $g_n^{(i)}=b_i^{(n)}$ and, for each $0\leq m\leq n-1$,
\begin{equation} \label{eq gn-1} 
 g_m^{(i)}=-\sum_{\ell=1}^rb_{\ell}^{(m)} \left( \sum_{k={m+1}}^n \varepsilon \left( g_k^{(i)} N_m^{k}(a_{\ell})\right)\right).
\end{equation}
Then, by \eqref{eq bi ai}, for all $1\leq i,j\leq r$,
\begin{equation} \label{eq x}
\varepsilon \left(g_n^{(i)} \sigma^{n}(a_j) \right) = \varepsilon\left( b^{(n)}_{i} \sigma^n(a_j)\right) = \delta_{ij}
\end{equation}
and
\[
\begin{split}
\varepsilon \left( g_m^{(i)} N_m^m(a_j) \right) &= \varepsilon \left( g_m^{(i)} \sigma^m(a_j)\right) \\
&\stackrel{\eqref{eq gn-1}}{=} \varepsilon \left( -\sum_{\ell=1}^r\left(b_{\ell}^{(m)} \left(\sum_{k={m+1}}^n \varepsilon \left( g_k^{(i)} N_m^{k}(a_{\ell})\right) \right)\right) \sigma^m(a_j)\right) \\
&= -\sum_{\ell=1}^r \sum_{k={m+1}}^n \varepsilon \left( g_k^{(i)} N_m^{k}(a_{\ell})\right) \varepsilon \left( b_{\ell}^{(m)} \sigma^m(a_j) \right) \\
&\stackrel{\eqref{eq x}}{=} - \sum_{k={m+1}}^n \varepsilon \left( g_k^{(i)} N_m^{k}(a_{j})\right).
\end{split}
\]
Hence
\begin{equation}\label{yy}
\sum_{k={m}}^n \varepsilon \left( g_k^{(i)} N_m^{k}(a_{j})\right) = 0
\end{equation}
for \(1 \leq i,j \leq r\), \(0 \leq m \leq n-1\). Now
\[
\begin{split} 
\alpha_\varepsilon(g^{(i)})(a_j) &= \alpha_{\varepsilon}(g^{(i)}a_j)(1)  \\
& \stackrel{\eqref{ga}}{=} \sum_{m= 0}^n\varepsilon \left(\sum_{k=m}^n g^{(i)}_k N_m^k(a)\right)x^m\\
& \stackrel{\eqref{yy},\eqref{eq x}}{=} x^na_i^*(a_j).
\end{split}
\]
So $x^na_i^* = \alpha_\varepsilon(g^{(i)}) \in \operatorname{Im} \alpha_{\varepsilon}$, as required.

Conversely, let $\alpha:S\to S^*$ be a right $S$-isomorphism. We would like to define $\varepsilon_\alpha:A\to \field$  as $\varepsilon_\alpha(a)=\alpha(a)(1)$ for $a \in A$. We need first to show that $\alpha(a)(1)\in \field$ for every $a\in A$.  Consider again the $\field$--basis $\{a_1,\ldots ,a_r\}$ of $A$, and set  $g_i = \alpha^{-1}(a_i^*)$ for $i = 1, \dots, r$. If we prove that the $\field$--linearly independent set $\{g_1, \dots, g_r \}$ is contained in $A$, then it becomes an $\field$--basis of $A$. 

Write $g_i=\sum_{k=0}^{n_i}g_{ik}x^k$. Therefore, 
\[
\begin{split}
\delta_{ij} &= a_i^*(a_j) \\
&= \alpha(g_i)(a_j) \\
&= \sum_{k=0}^{n_i}\alpha (g_{ik}x^k )(a_j) \\
&\overset{\ast}{=} \sum_{k=0}^{n_i}\alpha (g_{ik})(x^k a_j) \\
&= \sum_{k=0}^{n_i}\alpha (g_{ik}) \left( \sum_{m=0}^k N_m^k(a_j)x^m \right) \\
&\overset{\dagger}{=} \sum_{k=0}^{n_i}  \sum_{m=0}^k\alpha(g_{ik})(N_m^k(a_j))x^m,
\end{split}
\]
where equality $\ast$ comes from that $\alpha$ is a right $S$-morphism, and $\dagger$ is due to $\alpha(g_{ik})$ is a right $R$-morphism for every $k$ and $i$. Now, if $n_i\geq 1$, then
\[
\alpha(g_{i n_i})(\sigma^{n_i}(a_j))=0
\]
for every $j\in \{1,\ldots, r\}$. By Lemma \ref{lemma 1a},  $\{\sigma^{n_i}(a_1),...,\sigma^{n_i}(a_r)\}$ is a right $R$-basis of $S$, so $\alpha(g_{in_i})=0$ and then $g_{in_i}=0$. Hence, $n_i$ must be zero, so $g_i\in A$ for each $1\leq i\leq r$. Therefore, the $\field$--linear map $\alpha$ satisfies that $\alpha(g_i)(a_j) = \delta_{ij}$, for  the $\field$--bases $\{g_1, \dots, g_r \}$ and $\{a_1, \dots, a_r\}$ of $A$. This obviously implies that $\alpha(a)(b) \in \field$ for every $a, b \in A$, and that the bilinear form on $A$ given by $\langle a,b \rangle = \alpha(a)(b)$ is non-degenerate. Therefore, $\varepsilon_\alpha$ is a well defined Frobenius functional on $A$.

It remains to prove that both constructions are inverse one to each other. Indeed, let $\varepsilon$ be a Frobenius functional on $A$. Keeping the previous notation, for any $a\in A$, 
\[
\varepsilon_{\alpha_\varepsilon}(a)=\alpha_\varepsilon(a)(1)=\varepsilon(a).
\]
On the other hand, let $\alpha$ be a $S$-right isomorphism from $S$ to $S^*$. We want to check that $\alpha_{\varepsilon_\alpha}=\alpha$. Since both are right $S$--linear maps, it is enough if we prove that $\alpha_{\varepsilon_\alpha}(1) =\alpha (1)$. And these two maps are right $R$--linear, so that the following computation, for $a \in A$,  suffices:
\[
\alpha_{\varepsilon_\alpha}(1)(a)  = \varepsilon_\alpha(a) = \alpha (a)(1) = \alpha(1)(a).
\]
\end{proof}

Condition (1) in Theorem \ref{semiFrobext} is quite close to the notion of Frobenius extension, removing the need of being a left \(R\)-module morphism. We have not found in the literature that this condition has been introduced and studied. For this reason, let us now introduce semi Frobenius extensions. 

\begin{definition}
A unital ring extension $C \subseteq B$ is said to be right (resp. left) semi Frobenius if $B$ is a finitely generated projective right (resp. left) $C$-module and there exists an isomorphism \(B \cong B^*\) of right \(B\)-modules (resp. an isomorphism \(B \cong {}^*B\) of left \(B\)-modules).
\end{definition}

Our aim now is to prove that $A$ is a Frobenius algebra over $\field$ if and only if the extension $R \subseteq S$ is left or right semi Frobenius. 

\begin{theorem}\label{semiFrobtwosided}
Let \(A\) be an \(\field\)-algebra. The following statements are equivalent:
\begin{enumerate}
\item $A$ is a Frobenius $\field$-algebra,
\item the extension $R \subseteq S$ is right semi Frobenius,
\item the extension $R \subseteq S$ is left semi Frobenius.
\end{enumerate}
\end{theorem}

\begin{proof}
The equivalence between (1) and (2) is Theorem \ref{semiFrobext}. 

In order to check the equivalence (1) if and only if (3), observe that \(A\) is Frobenius if and only if \(A^{op}\) is Frobenius. By Theorem \ref{semiFrobext}, \(A^{op}\) is a Frobenius \(\field\)-algebra if and only if \(\field[x] \subseteq A^{op}[x;\sigma^{-1},-\delta \sigma^{-1}]\) is right semi Frobenius. Since \(\field[x] = R = R^{op}\) and \(S^{op} = A^{op}[x;\sigma^{-1},-\delta \sigma^{-1}]\) (see e.g. \cite[page 39, Exercise 2R]{Goodearl/Warfield:2004}), it follows that \(A^{op}\) is a Frobenius \(\field\)-algebra if and only if \(R \subseteq S\) is left semi Frobenius.
\end{proof}

\begin{remark}\label{leftright}
Although, by Theorems \ref{semiFrobext} and \ref{semiFrobtwosided}, $R\subseteq S$ is left semi Frobenius if and only if it is right semi Frobenius,  it is an open question to know if, in general, the notion of semi Frobenius extension is left-right symmetric, as it does for Frobenius extensions, see \cite[\S 1, page 11]{Nakayama/Tsuzuku:1959}.
\end{remark}

We now refine the latter results in the realm of Frobenius extensions. 

\begin{theorem}\label{bijectionFrob}
There exists a bijective correspondence between the sets of
\begin{enumerate}
\item \(R-S\)-isomorphisms from $S$ to $S^*$.
\item Frobenius functionals $\varepsilon:A\to \field$ satisfying $\varepsilon\sigma=\varepsilon$ and $\varepsilon\delta=0$.
\item Frobenius forms $\langle -,- \rangle :A\times A\to \field$ satisfying the conditions $\langle a,b\rangle=\langle\sigma(a),\sigma(b)\rangle$ and $\langle \sigma(a), \delta(b)\rangle +\langle\delta(a),b\rangle=0$ for all $a,b\in A$.
\end {enumerate}
\end{theorem}

\begin{proof}
In order to prove the bijection between \emph{(1)} and \emph{(2)}  it is enough to show that left $R$-linearity of a right $S$-isomorphism $\alpha:S\to S^*$ is equivalent to the conditions described in \emph{(2)} on the corresponding Frobenius functional $\varepsilon$ under the bijection stated in by Theorem \ref{semiFrobext}. 

Now, $\alpha$ is left $R$--linear if and only if $\alpha(xf) = x\alpha(f)$ for every $f \in S$. But, since $\alpha$ is right $S$--linear, the latter is equivalent to the condition $\alpha(x) = x\alpha(1)$. Both $\alpha(x)$ and $x\alpha(1)$ are right $R$--linear maps, so, they are equal if and only if $\alpha(x)(a) = (x\alpha(1))(a)$ for every $a \in A$.
Thus, from the computations 
\[
\alpha(x)(a) = \alpha(xa)(1) = \alpha(\sigma(a)x + \delta(a))(1) = \varepsilon(\sigma(a)) x + \varepsilon(\delta(a)),
\]
\[
(x\alpha(1))(a) = x \alpha(1)(a) = x \alpha(a)(1) = x \varepsilon(a) = \varepsilon(a)x,
\]
we get that $\alpha$ is left $R$--linear if and only if $\varepsilon(\sigma(a)) = \varepsilon(a)$ and $\varepsilon(\delta(a)) = 0$ for every $a \in A$.

The bijection between \emph{(2)} and \emph{(3)} follows from the bijection between Frobenius forms and Frobenius functionals explained in Remark \ref{FrobAlg}.
\end{proof}

The following direct consequence of Theorem \ref{bijectionFrob} is the characterization which will be used to build an example of biseparable extension which is not Frobenius. 

\begin{theorem}\label{Frobext}
$R \subseteq S$ is Frobenius if and only if  there exists a Frobenius functional $\varepsilon:A\to \field$ verifying  $\varepsilon\sigma=\varepsilon$ and $\varepsilon\delta=0$.
\end{theorem}

We finish the section showing a family of examples of left and right semi Frobenius,  but not Frobenius, extensions.

\begin{example}\label{seminoFrob} 
Let $p$ be a prime number and $\field_p$ the finite field of $p$ elements. Consider some $n>1$,  and the field extension $\field_p \subseteq \field_{p^n}$. Then $\field_{p^n}$ is a Frobenius $\field_{p}$-algebra. Let $\tau:\field_{p^n}\to \field_{p^n}$ be the Frobenius automorphism, i.e. $\tau(x)=x^p$ for any $x\in \field_{p^n}$. Then, there exists $\alpha\in \field_{p^n}$ such that $\{\alpha, \tau(\alpha),...,\tau^{n-1}(\alpha)\}$ is an $\field_{p}$-basis of $\field_{p^n}$. We set then the $\tau$-derivation $\delta:\field_{p^n}\rightarrow \field_{p^n}$ given by 
\[
\delta(b)=(\tau(b)-b)\frac{\alpha}{\tau(\alpha)-\alpha}
\]
for any $b\in \field_{p^n}$. By Theorem \ref{semiFrobtwosided}, $\field_{p}[x] \subseteq \field_{p^n}[x;\sigma,\delta]$ is left and right semi-Frobenius.  Nevertheless, it is not Frobenius. Indeed, by Theorem \ref{Frobext}, $\field_p[x] \subseteq \field_{p^n}[x;\tau,\delta]$ is Frobenius if and only if there exists a Frobenius functional $\varepsilon:\field_{p^n}\to \field_p$ such that $\varepsilon\tau=\varepsilon$ and $\varepsilon\delta=0$. But, in such a case, since $\delta(\alpha)=\alpha$, 
\[
0=\varepsilon(\delta(\alpha))=\varepsilon(\alpha)=\varepsilon(\tau(\alpha))=\cdots=\varepsilon(\tau^{n-1}(\alpha)).
\]
So that $\varepsilon=0$.
\end{example}

\section{Lifting biseparable extensions}\label{Biseparable}

In this section we aim to provide conditions for ensuring that the extension $R\subseteq S$ is biseparable. Since, by \cite[Lemma 3.3]{Caenepeel/Kadison:2001} and Lemma \ref{lemma 1a}, this is so if and only if  $R\subseteq S$ is separable and split, we deal with both notions independently. Let us first analyze the property of being split.

Let $C \subseteq B$ be a ring extension, $\sigma:B \to B$ an automorphism of $B$ and $\delta:B \to B$ a $\sigma$-derivation on $B$ such that $\sigma(C)\subseteq C$ and $\delta(C)\subseteq C$.

\begin{proposition}\label{split} 
 Suppose that $C \subseteq B$ is split and $\xi:B\to C$ is a \(C\)-bimodule morphism with  $\xi\sigma=\sigma\xi$, $\xi\delta=\delta \xi$ and $\xi(1)=1$, then $C[x;\sigma,\delta] \subseteq B[x;\sigma,\delta]$ is split. 
\end{proposition}

\begin{proof}
We define $\widehat{\xi}:B[x;\sigma, \delta] \to C[x;\sigma, \delta]$ as, for any $f=\sum_{i=0}^n b_ix^i\in B[x;\sigma,\delta]$,
\[
\widehat{\xi}(f )=\sum_{i=0}^n\xi(b_i)x^{i}.
\]
We check that $\widehat{\xi}$ is a $C[x;\sigma,\delta]$-bimodule morphism. Let  $a\in C$ and $f=\sum_{i=0}^nb_ix^{i}\in B[x;\sigma,\delta]$, 
\[
\begin{split}
\widehat{\xi}(xf) &= \widehat{\xi}\left (x\sum_{i=0}^nb_ix^{i}\right )\\
&= \widehat{\xi}\left (\sum_{i=0}^n\sigma(b_i)x^{i+1}+\delta(b_i)x^{i}\right )\\
&= \sum_{i=0}^n\xi(\sigma(b_i))x^{i+1}+\sum_{i=0}^n\xi(\delta(b_i))x^{i}\\
&= \sum_{i=0}^n\sigma(\xi(b_i))x^{i+1}+\sum_{i=0}^n\delta(\xi(b_i))x^{i}\\
&= x\sum_{i=0}^n\xi(b_i)x^{i}\\
&= x\widehat{\xi}(f),
\end{split}
\]
and
\[
\widehat{\xi} (af ) =\widehat{\xi}\left (\sum_{i=0}^nab_ix^{i}\right ) = \sum_{i=0}^n\xi(ab_i)x^{i} = a\sum_{i=0}^n\xi(b_i)x^{i} = a\widehat{\xi}(f),
\]
so $\widehat{\xi}$ is left $C[x;\sigma,\delta]$-linear. Analogously,
\[
\widehat{\xi}(fx) = \widehat{\xi}\left (\sum_{i=0}^nb_ix^{i+1}\right ) = \sum_{i=0}^n\xi(b_i)x^{i+1} = \left (\sum_{i=0}^n\xi(b_i)x^{i}\right )x = \widehat{\xi}(f)x,
\]
and
\[
\begin{split}
\widehat{\xi}(fa)
&= \widehat{\xi}\left (\sum_{i=0}^n b_i \left (\sum_{k=0}^{i}N_k^{i}(a)x^{k}\right) \right)\\
&= \sum_{i=0}^n\sum_{k=0}^{i}\xi(b_iN_k^{i}(a))x^{k}\\
&= \sum_{i=0}^n\sum_{k=0}^{i}\xi(b_i)N_k^{i}(a)x^{k}\\
&= \sum_{i=0}^n\xi(b_i)\left (\sum_{k=0}^{i}N_k^{i}(a)x^{k}\right )\\
&= \sum_{i=0}^n\xi(b_i)x^{i}a\\
&= \widehat{\xi}(f)a,\\
\end{split}
\]
so $\widehat{\xi}$ is right $C[x;\sigma,\delta]$-linear. Clearly, $\widehat{\xi}(1)=1$, and thus $B[x;\sigma,\delta] \subseteq C[x;\sigma,\delta]$ is split. 
\end{proof}

\begin{corollary}\label{corsplit}
If there exists an \(\field\)-linear  map \(\xi: A \to \field\) such that \(\xi(1) = 1\), \(\xi \sigma = \xi\) and \(\xi \delta = 0\), then \(R \subseteq S\) is split. 
\end{corollary}

\begin{proof}
Observe that any finite dimensional \(\field\)-algebra \(A\) is split, since there is an \(\field\)-basis of \(A\) containing the element \(1\).  Hence, the corollary follows from Proposition \ref{split}, since \(\sigma_{|\field} = \operatorname{id}_{\field}\) and \(\delta_{|\field} = 0\).
\end{proof}

The transfer of separability in Ore extensions is studied in \cite{Gomez/Lobillo/Navarro:2017a}. For brevity, we denote by $\sigma^\otimes$ and $\delta^\otimes$ the maps
\[
\begin{split}
\sigma^\otimes:B\otimes_C B &\to B\otimes_C B\\ 
b_1\otimes b_2 &\mapsto \sigma(b_1)\otimes \sigma(b_2)\\
& \\
\delta^\otimes:B\otimes_C B &\to B\otimes_C B\\ 
b_1\otimes b_2 &\mapsto \sigma(b_1)\otimes \delta(b_2)+\delta(b_1)\otimes b_2
\end{split}
\]
for every $b_1,b_2\in B$. By \cite[Lemma 27]{Gomez/Lobillo/Navarro:2017a}, $\sigma^\otimes$ and $\delta^\otimes$ are well defined.

We will use the following proposition whose easy proof compares $xp$ and $px$ in the light of the rule defining the product in an Ore extension.

\begin{proposition}[{\cite[Theorem 29]{Gomez/Lobillo/Navarro:2017a}}]\label{sep}
 If $C \subseteq B$ is separable and there exists a separability element $p$ verifying $\sigma^\otimes(p)=p$ and  $\delta^\otimes(p)=0$, then $C[x;\sigma,\delta] \subseteq B[x;\sigma,\delta]$ is separable. 
\end{proposition}

In \cite[Theorem 8]{Gomez/Lobillo/Navarro:2017b} a converse to Proposition \ref{sep} is provided when \(\delta = 0\). Here, we generalize part of this result when \(\delta\) is an inner \(\sigma\)--derivation. So, for the rest of this section, \(\sigma : A \to A\) is an \(\field\)--linear automorphism and \(\delta_{\sigma,b} : A \to A\) is a $\sigma$-derivation defined by 
\[
\delta_{\sigma,b}(a) = ba - \sigma(a) b
\]
for some \(b \in A\). Hence \(R = \field[x]\) and \(S = A[x;\sigma,\delta_{\sigma,b}]\). Recall that we have fixed an \(\field\)--basis \(\{a_1, \dots, a_r\}\) of \(A\). 

\begin{lemma}\label{gradingtensor}
The set \(\{a_i \otimes_R a_j x^k ~|~ 1 \leq i,j \leq r, k \geq 0\}\) is an \(\field\)--basis of \(S \otimes_R S\). Consequently, the map
\[
\varphi : S \otimes_R S \to \bigoplus_{k \geq 0} (A \otimes_\field A) x^k, \quad \left[a_i \otimes_R a_j x^k \mapsto (a_i \otimes_\field a_j) x^k \right]
\]
is an \(\field\)--isomorphism that provides an \(\mathbb{N}\)--grading on \(S \otimes_R S\) as an \(\field\)--vector space. 
\end{lemma}

\begin{proof}
It can be derived from Lemma \ref{lemma 1a} and \cite[Corollary 8.5]{Stenstrom:1975} that \(S \otimes_R S\) is a free right \(R\)-module with basis \(\{a_i \otimes_R a_j ~|~ 1 \leq i,j \leq r\}\), hence \(\{a_i \otimes_R a_j x^k ~|~ 1 \leq i,j \leq r, k \geq 0\}\) is an \(\field\)--basis. Consequently \(\varphi\) is an isomorphism because \(\{(a_i \otimes_\field a_j) x^k ~|~ 1 \leq i,j \leq r, k \geq 0\}\) is a basis of \(\bigoplus_{k \geq 0} (A \otimes_\field A) x^k\). 
\end{proof}

\begin{proposition}\label{sepdown}
If \(R \subseteq S\) is separable and \(\delta = \delta_{\sigma,b}\) is inner, then \(\field \subseteq A\) is separable. 
\end{proposition}

\begin{proof}
Let \(p \in S \otimes_R S\) be a separability element. We do not lose generality if we assume \(p = \sum_{i=1}^r \sum_{j=0}^m a_i \otimes_R g_{ij} x^j\). Let \(a \in A\), where $\{a_1, \dots, a_r \}$ is an $\field$--basis of $A$.  Since \(ap = pa\) we have
\[
\begin{split}
\sum_{i=1}^r \sum_{j=0}^m aa_i \otimes_R g_{ij} x^j & = \sum_{i=1}^r \sum_{j=0}^m \sum_{k=0}^j a_i \otimes_R g_{ij} N_k^j(a) x^k \\
& = \sum_{i=1}^r \sum_{k=0}^m \sum_{j=k}^m a_i \otimes_R g_{ij} N_k^j(a) x^k.
\end{split}
\]
By Lemma \ref{gradingtensor} and by applying \(\varphi\), we get that, for all \(0 \leq \ell \leq m\),
\[
\sum_{i=1}^r aa_i \otimes_\field g_{i \ell} = \sum_{i=1}^r \sum_{j=\ell}^m a_i \otimes_\field g_{ij} N_\ell^j(a) \in A \otimes_\field A.
\]
Multiplying on the right by \(b^\ell\) and adding all the obtained identities we have
\[
\begin{split}
\sum_{\ell=0}^m \sum_{i=1}^r aa_i \otimes_\field g_{i \ell} b^\ell & = \sum_{\ell=0}^m \sum_{i=1}^r \sum_{j=\ell}^m a_i \otimes_\field g_{ij} N_\ell^j(a) b^\ell \\
& = \sum_{i=1}^r \sum_{j=0}^m \sum_{\ell=0}^j a_i \otimes_\field g_{ij} N_\ell^j(a) b^\ell. 
\end{split}
\]
Since \(ba = \sigma(a) b + \delta_{\sigma,b}(a)\), an inductive argument on $j$, which uses \eqref{Nnmas1}, shows that 
\[
\sum_{\ell=0}^j N_\ell^j(a) b^\ell = b^j a, 
\]
hence
\[
\sum_{i=1}^r \sum_{\ell=0}^m aa_i \otimes_\field g_{i \ell} b^\ell = \sum_{i=1}^r \sum_{j=0}^m a_i \otimes_\field g_{ij} b^j a.
\]
So \(\widehat{p} = \sum_{i=1}^r \sum_{j=0}^m a_i \otimes_\field g_{ij} b^j\) satisfies \(a \widehat{p} = \widehat{p} a\) for all \(a \in A\).  Now, since 
\[ 
1 = \mu(p) = \sum_{i=1}^r \sum_{j=0}^m a_{i} g_{ij} x^j \in A[x;\sigma,\delta_{\sigma,b}],
\]
it follows that \(1 = \sum_{i=1}^r a_i g_{i0}\) and \(0 = \sum_{i=1}^r a_i g_{ij}\) for all \(1 \leq j \leq m\). Therefore \(\mu(\widehat{p}) = 1\) and \(\widehat{p}\) is a separability element for \(\field \subseteq A\). 
\end{proof}

\section{An answer to a problem of Caenepeel and Kadison}\label{sec:Example}

In this section, with the aid of the results of Sections \ref{Frobenius} and \ref{Biseparable}, we give a negative answer to Problem \ref{theproblem}.

\begin{example}[Answer to Problem \ref{theproblem}]\label{counterexample}
Let $\field_8$ be the field with eight elements described as $\field_8=\field_2(a)$, where $a^3+a^2+1=0$. Let $\tau$ be the Frobenius automorphism on $\field_{8}$, that is, $\tau(c)=c^2$ for every $c\in \field_8$.  Observe that $\{a, a^2,a^4\}$ is an auto dual basis of the extension $\field_2 \subseteq \field_8$.  Set $A=\mathcal{M}_2(\field_8)$, the ring of $2\times 2$ matrices over $\field_8$, and consider the $\field_2$-automorphism $\sigma: A \to A$ defined as the component-by-component extension of $\tau$ to $A$. That is, $\sigma$ is given by
\begin{equation}\label{automorphism}
\sigma\left ( \begin{matrix} x_0 & x_1 \\ x_2 & x_3 \end{matrix} \right ) = \left ( \begin{matrix} \tau(x_0) & \tau(x_1) \\ \tau(x_2) & \tau(x_3) \end{matrix} \right ) = \left ( \begin{matrix} x_0^2 & x_1^2 \\ x_2^2 & x_3^2 \end{matrix} \right ) \text{ for every } \left ( \begin{matrix} x_0 & x_1 \\ x_2 & x_3 \end{matrix} \right )\in A.
\end{equation}
We can also set the inner $\sigma$-derivation $\delta:A\rightarrow A$ given by $\delta(X)=MX-\sigma(X)M$ for $X\in A$, where 
\[
M=\begin{pmatrix} 
0 & 0 \\
0 & a 
\end{pmatrix}.
\]
Our aim is to prove that the ring extension $\field_2[x] \subseteq A[x;\sigma, \delta]$ is split and separable, and hence biseparable, but not Frobenius. For simplicity, we denote
\[
e_0= \begin{pmatrix} 
1 & 0 \\
0 & 0 
\end{pmatrix}, \, e_1= \begin{pmatrix} 
0 & 1 \\
0 & 0 
\end{pmatrix}, \, e_2= \begin{pmatrix} 
0& 0 \\
1 & 0 
\end{pmatrix} \text{ and } e_3= \begin{pmatrix} 
0& 0 \\
0 & 1 
\end{pmatrix}.
\]
Hence, an $\field_2$-basis of $A$ is given by $\mathcal{B}=\{a^{2^i}e_j \text{ such that } 0\leq i \leq 2 \text{ and } 0\leq j \leq 3\}$. 

Let $\varepsilon: A \to \field_2$ be an $\field_2$-linear map. If we force $\varepsilon \sigma = \varepsilon$, then 
\[
\varepsilon(a^{2^{i+1}}e_j)=\varepsilon\sigma(a^{2^i}e_j)=\varepsilon(a^{2^{i}}e_j)
\]
for every $0 \leq i \leq 2$ and \(0 \leq j \leq 3\), so that  $\varepsilon$ is determined by four values $\gamma_0,\gamma_1,\gamma_2, \gamma_3\in \field_2$ such that $\varepsilon(a^{2^i}e_j)=\gamma_j$ for  $0 \leq i \leq 2$ and \(0 \leq j \leq 3\).

Let us then consider $\xi:A\rightarrow \field_2$ the $\field_2$-linear map determined by $\gamma_0=1$,  $\gamma_1=0$, $\gamma_2=0$ and $\gamma_3=0$. Firstly,
\[
\begin{split}
\xi\begin{pmatrix} 
1& 0 \\
0 & 1 
\end{pmatrix}
& =\xi\begin{pmatrix} 
a+a^2+a^4& 0 \\
0 & a+a^2+a^4 
\end{pmatrix}\\
& =\xi(ae_0)+\xi(a^2e_0)+\xi(a^4e_0)+\xi(ae_3)+\xi(a^2e_3)+\xi(a^4e_3)\\
& =1.
\end{split}
\]
On the other hand, for any $x_0,x_1,x_2,x_3 \in \field_8$,
\begin{equation}\label{derivation}
\begin{split}
\delta\begin{pmatrix} 
x_0 & x_1 \\
x_2 & x_3 
\end{pmatrix} &= \begin{pmatrix} 
0 & 0 \\
0 & a 
\end{pmatrix} \begin{pmatrix} 
x_0 & x_1 \\
x_2 & x_3 
\end{pmatrix} + \begin{pmatrix} 
x_0^2 & x_1^2 \\
x_2^2 & x_3^2 
\end{pmatrix}\begin{pmatrix} 
0 & 0 \\
0 & a 
\end{pmatrix}\\ 
&= \begin{pmatrix} 
0 & 0 \\
ax_2 & ax_3 
\end{pmatrix} + \begin{pmatrix} 
0 & ax_1^2 \\
0 & ax_3^2 
\end{pmatrix}\\
&= \begin{pmatrix} 
0 & ax_1^2  \\
ax_2 & a(x_3+x_3^2). 
\end{pmatrix}
\end{split}
\end{equation}
Therefore, $\xi\delta = 0$. By Corollary \ref{corsplit}, the extension $\field_2[x] \subseteq A[x;\sigma, \delta]$ is split.

Let us prove that the map $\xi$ is the only non trivial $\field_2$-linear map verifying the equalities $\xi\sigma=\xi$ and $\xi\delta=0$. Let us suppose that $\varepsilon:A\rightarrow \field_2$ is a non zero $\field_2$-linear map that verifies the equation $\varepsilon\sigma=\varepsilon$. As reasoned above, it is determined by  some values  $\gamma_0,\gamma_1,\gamma_2, \gamma_3\in \field_2$. Nevertheless,
\begin{itemize}
\item If $\gamma_1=1$, then
$\varepsilon\delta\begin{pmatrix} 
0& 1 \\
0 & 0 
\end{pmatrix} = \varepsilon \begin{pmatrix} 
0& a \\
0 & 0 
\end{pmatrix}=1$,
\item If $\gamma_2=1$, then
$\varepsilon\delta\begin{pmatrix} 
0& 0 \\
1 & 0 
\end{pmatrix} = \varepsilon \begin{pmatrix} 
0& 0 \\
a & 0 
\end{pmatrix}=1$,
\item If $\gamma_3=1$, then
$\varepsilon\delta\begin{pmatrix} 
0& 0 \\
0 & a 
\end{pmatrix} = \varepsilon \begin{pmatrix} 
0& 0 \\
0 & a^2 + a^3
\end{pmatrix}= \varepsilon \begin{pmatrix} 
0& 0 \\
0 & a+a^2+a^4 
\end{pmatrix}=1$,
\end{itemize}
so that $\varepsilon\delta=0$ implies $\gamma_1=\gamma_2=\gamma_3=0$. Hence, $\gamma_0 = 1$, and $\varepsilon = \xi$.  Note that the kernel of $\xi$ contains the left ideal 
\[
J=\left\{\begin{pmatrix} 
0& c_2\\
0 & c_3
\end{pmatrix} \mid c_2,c_3\in \field_8 \right\},
\]
so that there is no Frobenius functional $\varepsilon:A\to \field_2$ verifying $\varepsilon\sigma=\varepsilon$ and $\varepsilon\delta=0$. By Corollary \ref{Frobext}, the extension $\field_2[x] \subseteq A[x;\sigma, \delta]$ is not Frobenius.

Finally, let us prove that the extension is separable. Consider the element $p\in A\otimes_{\field_2}A$ given by
\[
\begin{split}
p &= \begin{pmatrix} 
a& 0 \\
0 & 0 
\end{pmatrix} \otimes \begin{pmatrix} 
a& 0 \\
0 & 0 
\end{pmatrix} + \begin{pmatrix} 
a^2& 0 \\
0 & 0 
\end{pmatrix} \otimes \begin{pmatrix} 
a^2& 0 \\
0 & 0 
\end{pmatrix} + \begin{pmatrix} 
a^4& 0 \\
0 & 0 
\end{pmatrix} \otimes \begin{pmatrix} 
a^4& 0 \\
0 & 0 
\end{pmatrix} \\
&\quad + \begin{pmatrix} 
0& 0 \\
a & 0 
\end{pmatrix} \otimes \begin{pmatrix} 
0& a\\
0 & 0 
\end{pmatrix} + \begin{pmatrix} 
0& 0 \\
a^2 & 0 
\end{pmatrix} \otimes \begin{pmatrix} 
0& a^2 \\
0 & 0 
\end{pmatrix} + \begin{pmatrix} 
0& 0 \\
a^4 & 0 
\end{pmatrix}\otimes \begin{pmatrix} 
0& a^4\\
0 & 0 
\end{pmatrix}.
\end{split}
\]
This is a separability element of the extension $\field_2 \subseteq A$, since it is the composition of the separability element $a \otimes a + a^2 \otimes a^2 + a^4 \otimes a^4$ of the extension $\field_2 \subseteq \field_8$, and the separability element $e_0 \otimes e_0 + e_2 \otimes e_3$ of the extension $\field_8 \subseteq A$, see \cite[Examples 4 and 5]{Gomez/Lobillo/Navarro:2017a} and \cite[Proposition 2.5]{Hirata/Sugano:1966}. Although it is straightforward to check that $\sigma^{\otimes}(p)=p$ and $\delta^{\otimes}(p)=0$, due to its importance in this paper, we detail explicitly all the computations. Since the Frobenius automorphism induces a permutation on \(\{a,a^2,a^4\}\), it follows that
\[
\begin{split}
\sigma^\otimes(p) &=\begin{pmatrix} 
a^2 & 0 \\
0 & 0 
\end{pmatrix} \otimes \begin{pmatrix} 
a^2 & 0 \\
0 & 0 
\end{pmatrix} + \begin{pmatrix} 
a^4 & 0 \\
0 & 0 
\end{pmatrix} \otimes \begin{pmatrix} 
a^4 & 0 \\
0 & 0 
\end{pmatrix} + \begin{pmatrix} 
a & 0 \\
0 & 0 
\end{pmatrix} \otimes \begin{pmatrix} 
a & 0 \\
0 & 0 
\end{pmatrix} \\
&\quad + \begin{pmatrix} 
0 & 0 \\
a^2 & 0 
\end{pmatrix} \otimes \begin{pmatrix} 
0 & a^2 \\
0 & 0 
\end{pmatrix} + \begin{pmatrix} 
0 & 0 \\
a^4 & 0 
\end{pmatrix} \otimes \begin{pmatrix} 
0 & a^4 \\
0 & 0 
\end{pmatrix} + \begin{pmatrix} 
0 & 0 \\
a & 0 
\end{pmatrix}\otimes \begin{pmatrix} 
0 & a\\
0 & 0 
\end{pmatrix} \\
&= p.
\end{split}
\]
Let us now compute \(\delta^\otimes(p)\). Recall \(\delta^\otimes = \sigma \otimes \delta + \delta \otimes \operatorname{id}\). By \eqref{derivation} and \eqref{automorphism}, \(\delta \left( \begin{smallmatrix} c & 0 \\ 0 & 0 \end{smallmatrix} \right) = \left( \begin{smallmatrix} 0 & 0 \\ 0 & 0 \end{smallmatrix} \right)\) for each \(c \in \field_8\), so
\[
\delta^\otimes \left(\begin{pmatrix} 
a^{2^i} & 0 \\
0 & 0 
\end{pmatrix} \otimes \begin{pmatrix} 
a^{2^i} & 0 \\
0 & 0 
\end{pmatrix}\right) = \begin{pmatrix} 
a^{2^{i+1}} & 0 \\
0 & 0 
\end{pmatrix} \otimes \begin{pmatrix} 
0 & 0 \\
0 & 0 
\end{pmatrix} + \begin{pmatrix} 
0 & 0 \\
0 & 0 
\end{pmatrix} \otimes \begin{pmatrix} 
a^{2^i} & 0 \\
0 & 0 
\end{pmatrix},
\]
for \(0 \leq i \leq 2\). Hence
\begin{equation}\label{deltafirstreduction}
\begin{split}
\delta^\otimes(p) &= \delta^\otimes \left(\begin{pmatrix} 
a& 0 \\
0 & 0 
\end{pmatrix} \otimes \begin{pmatrix} 
a& 0 \\
0 & 0 
\end{pmatrix}\right) 
+ \delta^\otimes \left(\begin{pmatrix} 
a^2& 0 \\
0 & 0 
\end{pmatrix} \otimes \begin{pmatrix} 
a^2& 0 \\
0 & 0 
\end{pmatrix}\right) \\
&\quad + \delta^\otimes \left(\begin{pmatrix} 
a^4& 0 \\
0 & 0 
\end{pmatrix} \otimes \begin{pmatrix} 
a^4& 0 \\
0 & 0 
\end{pmatrix}\right) 
+ \delta^\otimes \left(\begin{pmatrix} 
0& 0 \\
a & 0 
\end{pmatrix} \otimes \begin{pmatrix} 
0& a\\
0 & 0 
\end{pmatrix}\right) \\
&\quad + \delta^\otimes \left(\begin{pmatrix} 
0& 0 \\
a^2 & 0 
\end{pmatrix} \otimes \begin{pmatrix} 
0& a^2 \\
0 & 0 
\end{pmatrix}\right) 
+ \delta^\otimes \left(\begin{pmatrix} 
0& 0 \\
a^4 & 0 
\end{pmatrix}\otimes \begin{pmatrix} 
0& a^4\\
0 & 0 
\end{pmatrix}\right) \\
&= \delta^\otimes \left(\begin{pmatrix} 
0& 0 \\
a & 0 
\end{pmatrix} \otimes \begin{pmatrix} 
0& a\\
0 & 0 
\end{pmatrix}\right) \\
&\quad + \delta^\otimes \left(\begin{pmatrix} 
0& 0 \\
a^2 & 0 
\end{pmatrix} \otimes \begin{pmatrix} 
0& a^2 \\
0 & 0 
\end{pmatrix}\right) 
+ \delta^\otimes \left(\begin{pmatrix} 
0& 0 \\
a^4 & 0 
\end{pmatrix}\otimes \begin{pmatrix} 
0& a^4\\
0 & 0 
\end{pmatrix}\right)
\end{split}
\end{equation}
Moreover, by \eqref{derivation} and \eqref{automorphism} again,
\[
\delta^\otimes \left(\begin{pmatrix} 
0 & 0 \\
a^{2^i} & 0 
\end{pmatrix} \otimes \begin{pmatrix} 
0 & a^{2^i} \\
0 & 0 
\end{pmatrix}\right) = \begin{pmatrix} 
0 & 0 \\
a^{2^{i+1}} & 0 
\end{pmatrix} \otimes \begin{pmatrix} 
0 & a^{2^{i+1}+1} \\
0 & 0 
\end{pmatrix} + \begin{pmatrix} 
0 & 0 \\
a^{2^i+1} & 0 
\end{pmatrix} \otimes \begin{pmatrix} 
0 & a^{2^i} \\
0 & 0 
\end{pmatrix},
\]
so we can follow the computations in \eqref{deltafirstreduction} to get 
\begin{equation}\label{deltasecondreduction}
\begin{split}
\delta^\otimes(p) &= \begin{pmatrix} 
0 & 0 \\
a^{2} & 0 
\end{pmatrix} \otimes \begin{pmatrix} 
0 & a^{3} \\
0 & 0 
\end{pmatrix} + \begin{pmatrix} 
0 & 0 \\
a^{2} & 0 
\end{pmatrix} \otimes \begin{pmatrix} 
0 & a \\
0 & 0 
\end{pmatrix} \\
&\quad + \begin{pmatrix} 
0 & 0 \\
a^{4} & 0 
\end{pmatrix} \otimes \begin{pmatrix} 
0 & a^{5} \\
0 & 0 
\end{pmatrix} + \begin{pmatrix} 
0 & 0 \\
a^{3} & 0 
\end{pmatrix} \otimes \begin{pmatrix} 
0 & a^{2} \\
0 & 0 
\end{pmatrix} \\
&\quad + \begin{pmatrix} 
0 & 0 \\
a & 0 
\end{pmatrix} \otimes \begin{pmatrix} 
0 & a^{2} \\
0 & 0 
\end{pmatrix} + \begin{pmatrix} 
0 & 0 \\
a^{5} & 0 
\end{pmatrix} \otimes \begin{pmatrix} 
0 & a^{4} \\
0 & 0 
\end{pmatrix},
\end{split}
\end{equation}
where we have used that \(a^7 = 1\). The identities $a^3 = a + a^4$ and $a^5 = a^2 + a^4$ in \(\field_8\)
allow us to expand \eqref{deltasecondreduction} in order obtain
\begin{equation}\label{deltathirdreduction}
\begin{split}
\delta^\otimes(p) &= \begin{pmatrix} 
0 & 0 \\
a^{2} & 0 
\end{pmatrix} \otimes \begin{pmatrix} 
0 & a + a^4 \\
0 & 0 
\end{pmatrix} + \begin{pmatrix} 
0 & 0 \\
a^{2} & 0 
\end{pmatrix} \otimes \begin{pmatrix} 
0 & a \\
0 & 0 
\end{pmatrix} \\
&\quad + \begin{pmatrix} 
0 & 0 \\
a^{4} & 0 
\end{pmatrix} \otimes \begin{pmatrix} 
0 & a^2 + a^4 \\
0 & 0 
\end{pmatrix} + \begin{pmatrix} 
0 & 0 \\
a + a^4 & 0 
\end{pmatrix} \otimes \begin{pmatrix} 
0 & a^{2} \\
0 & 0 
\end{pmatrix} \\
&\quad + \begin{pmatrix} 
0 & 0 \\
a & 0 
\end{pmatrix} \otimes \begin{pmatrix} 
0 & a^{2} \\
0 & 0 
\end{pmatrix} + \begin{pmatrix} 
0 & 0 \\
a^2 + a^4 & 0 
\end{pmatrix} \otimes \begin{pmatrix} 
0 & a^{4} \\
0 & 0 
\end{pmatrix} \\
&= \begin{pmatrix} 
0 & 0 \\
a^{2} & 0 
\end{pmatrix} \otimes \begin{pmatrix} 
0 & a \\
0 & 0 
\end{pmatrix} 
+ \begin{pmatrix} 
0 & 0 \\
a^{2} & 0 
\end{pmatrix} \otimes \begin{pmatrix} 
0 & a^4 \\
0 & 0 
\end{pmatrix} \\
&\quad + \begin{pmatrix} 
0 & 0 \\
a^{2} & 0 
\end{pmatrix} \otimes \begin{pmatrix} 
0 & a \\
0 & 0 
\end{pmatrix} 
+ \begin{pmatrix} 
0 & 0 \\
a^{4} & 0 
\end{pmatrix} \otimes \begin{pmatrix} 
0 & a^2 \\
0 & 0 
\end{pmatrix} \\
&\quad + \begin{pmatrix} 
0 & 0 \\
a^{4} & 0 
\end{pmatrix} \otimes \begin{pmatrix} 
0 & a^4 \\
0 & 0 
\end{pmatrix} 
+ \begin{pmatrix} 
0 & 0 \\
a & 0 
\end{pmatrix} \otimes \begin{pmatrix} 
0 & a^{2} \\
0 & 0 
\end{pmatrix} \\
&\quad + \begin{pmatrix} 
0 & 0 \\
a^4 & 0 
\end{pmatrix} \otimes \begin{pmatrix} 
0 & a^{2} \\
0 & 0 
\end{pmatrix}
+ \begin{pmatrix} 
0 & 0 \\
a & 0 
\end{pmatrix} \otimes \begin{pmatrix} 
0 & a^{2} \\
0 & 0 
\end{pmatrix}  \\
&\quad + \begin{pmatrix} 
0 & 0 \\
a^2 & 0 
\end{pmatrix} \otimes \begin{pmatrix} 
0 & a^{4} \\
0 & 0 
\end{pmatrix}
+ \begin{pmatrix} 
0 & 0 \\
a^4 & 0 
\end{pmatrix} \otimes \begin{pmatrix} 
0 & a^{4} \\
0 & 0 
\end{pmatrix}\\
& = 0.
\end{split}
\end{equation}
By Proposition \ref{sep}, $\field_2[x] \subseteq A[x;\sigma, \delta]$ is separable. Hence $\field_2[x] \subseteq A[x;\sigma, \delta]$ is a biseparable extension which is not Frobenius.
\end{example}

At this point one could ask what happens if we replace the family of Frobenius extensions in Problem \ref{theproblem} by a more general family. For instance, we can consider the family of Frobenius extensions of second kind introduced in \cite{Nakayama/Tsuzuku:1960}. Let $C \subseteq B$ be a ring extension and let $\kappa: C \rightarrow C$ be an automorphism. There is a structure of left \(C\)-module on \(C\) given by $a \cdot_{\kappa} b = \kappa(a)b$ for each $a,b \in C$. Hence, $C \subseteq B$ is said to be a $\kappa$-Frobenius extension, or a Frobenius extension of second kind, if  $B$ is a finitely generated projective right $C$-module, and there exists a \(C-B\)-isomorphism from $B$ to $B^{*_\kappa}=\operatorname{Hom}(B_C, {_\kappa}C_C)$. The \(C-B\)-bimodule structure on $B^{*_\kappa}$ is then given by $(afb)(c)=a\cdot_{\kappa}f(bc)=\kappa(a)f(bc)$ for any $f\in B^{*_\kappa}$, $a\in C$ and $b,c\in B$. It is clear that a Frobenius extension of second kind is left and right semi Frobenius. A natural question that arises is then if a biseparable extension is a Frobenius extension of second kind. In order to answer this question, we may prove similar results to those showed in the previous sections. 

\begin{proposition}\label{prop second kind}
Let $\kappa:R\to R$ be an automorphism with $\kappa(x)=mx+n$ for some $m,n\in \field$ with $m\not =0$. There exists a bijection between the sets of
\begin{enumerate}
\item \(R-S\)-isomorphisms $\alpha:S\to S^{*_\kappa}$.
\item Frobenius functionals $\varepsilon:A\to \field$ verifying $\varepsilon\sigma=m\varepsilon$ and $\varepsilon\delta=n\varepsilon$.
\end{enumerate}
\end{proposition}

\begin{proof}
By Theorem \ref{semiFrobext}, there exists a right $S$-isomorphism $\beta:S\to S^{*_\kappa}$ if and only if there exists a Frobenius functional $\varepsilon:A\rightarrow \field$. Now, analogously to the proof of Theorem \ref{Frobext},
\[
\kappa(x)\beta(1)(a)=m\varepsilon(a) x+n\varepsilon(a).
\]
and
\[
\beta(x)(a)=\beta(1)(xa)=\beta(1) (\sigma(a)x+\delta(a))=\varepsilon(\sigma(a))x+\varepsilon(\delta(a))
\]
for any $a\in A$. Hence,  $\beta$ is left $R$-linear if and only if $\varepsilon \sigma = m\varepsilon$ and $\varepsilon \delta = n\varepsilon$.
\end{proof}

\begin{corollary}\label{Frob2KExt}
$R \subseteq S$ is a Frobenius extension of second kind if and only if there exists a Frobenius functional $\varepsilon: A \to \field$ and $m,n\in\field$ with $m \neq 0$ such that $\varepsilon \sigma = m\varepsilon$ and $\varepsilon \delta = n\varepsilon$.
\end{corollary}

Are biseparable extensions Frobenius extensions of second kind? The answer is again negative. 

\begin{example}[Biseparable extensions are not necessarily Frobenius of second kind] By the latter result, Example \ref{counterexample} also provides an example of a biseparable extension which is not Frobenius of second kind. Indeed, let $\kappa:\field_2[x]\to \field_2[x]$ be an automorphism. Hence $\kappa(x)=x+n$ for some $n\in \field_2$. The case $n=0$ is already analyzed in Example \ref{counterexample}. Therefore, set $\kappa(x)=x+1$. By Proposition \ref{prop second kind},  $\field_2[x] \subseteq A[x;\sigma, \delta]$ is Frobenius of second kind if and only if there exists a Frobenius functional $\varepsilon:A\to \field_2$ verifying $\varepsilon\sigma=\varepsilon$ and $\varepsilon\delta=\varepsilon$. As reasoned in Example \ref{counterexample}, $\varepsilon$ is determined by four values $\gamma_0,\gamma_1,\gamma_2, \gamma_3\in \field_2$ such that $\varepsilon(a^{2^i}e_j)=\gamma_j$ for any $i=0,1,2$ and $j=0,1,2,3$. Now,
\begin{itemize}
\item If $\gamma_0=1$, then
$0=\varepsilon\delta\begin{pmatrix} 
a& 0 \\
0 & 0 
\end{pmatrix} \not = \varepsilon \begin{pmatrix} 
a& 0 \\
0 & 0 
\end{pmatrix}=1$,
\item If $\gamma_1=1$, then
$0=\varepsilon\delta\begin{pmatrix} 
0& a \\
0 & 0 
\end{pmatrix} \not = \varepsilon \begin{pmatrix} 
0& a \\
0 & 0 
\end{pmatrix}=1$,
\item If $\gamma_2=1$, then
$0=\varepsilon\delta\begin{pmatrix} 
0& 0 \\
a^2 & 0 
\end{pmatrix} \not = \varepsilon \begin{pmatrix} 
0& 0 \\
a^2 & 0 
\end{pmatrix}=1$,
\item If $\gamma_3=1$, then
$0=\varepsilon\delta\begin{pmatrix} 
0& 0 \\
0 & a^2 
\end{pmatrix} \not = \varepsilon \begin{pmatrix} 
0& 0 \\
0 & a^2 
\end{pmatrix}=1$,
\end{itemize}
so that $\varepsilon\delta=\varepsilon$ if and only if $\varepsilon=0$. By Corollary \ref{Frob2KExt}, $\field_2[x] \subseteq A[x;\sigma, \delta]$ is not Frobenius of second kind. Additionally, we may state that the class of Frobenius extensions of second kind is strictly contained in the class of left and right semi Frobenius.
\end{example}

We can formulate the next problem. 

\begin{problem}\label{extendedproblem}
Are biseparable extensions left and right semi Frobenius?
\end{problem}

The techniques developed in this paper are not suitable to handle with this problem. In fact, assume \(R \subseteq S\) is biseparable with \(\delta = \delta_{\sigma,b}\) inner. Then \(\field \subseteq A\) is separable by Proposition \ref{sepdown}. By \cite[Proposition 5]{Eilenberg/Nakayama:1955} or \cite[Theorem 4.2]{Endo/Watanabe:1967}, \(\field \subseteq A\) is a Frobenius extension, hence \(R \subseteq S\) is right and left semi Frobenius by Theorem \ref{semiFrobtwosided}.

\end{document}